\documentclass[a4paper]{amsart}
\usepackage[T1]{fontenc}
\usepackage{graphicx}
\usepackage{verbatim}
\usepackage{tikz} 
\usepackage{enumerate}
\usepackage{pdfsync}
\usetikzlibrary{calc} 

\newtheorem{theorem}{Theorem}

\newtheorem{lemma}[theorem]{Lemma}
\newtheorem{proposition}[theorem]{Proposition}
\newtheorem{corollary}[theorem]{Corollary}
\newtheorem{definition}{Definition}

\theoremstyle{remark}
\newtheorem*{remark}{Remark}

\newcommand{\E}[1]{\mathbf{E}_{\Lambda}\left[#1\right]}

\def\P{{\mathbf P}}
\def\Pr{\P}
\newcommand{\indicator}[1]{\mathbf{1}_{ {#1} }}
\newcommand{\stirling}[2]{\left\{\mbox{\hspace{-1.5mm}}\begin{array}{c}#1\\ #2\end{array}\mbox{\hspace{-1.5mm}}\right\}}
\def\d{{\text{ d}}}
\def\esp#1{\E#1}
\def\R{{\mathbb R}}
\def\N{{\mathbb N}}
\newcommand{\opcov}{\operatorname{Cov}}
\newcommand{\opvar}{\operatorname{Var}}
\newcommand{\dom}{\operatorname{Dom}}

\newcommand{\var}[1]{\opvar_{\Lambda}\left[#1\right]}
\newcommand{\cov}[1]{\opcov_{\Lambda}\left[#1\right]}
\renewcommand\SS{{\mathfrak S}}
\newcommand\U{{\mathcal U}}

\begin{document}

\title{Simplicial Homology of Random Configurations}
\author{L. Decreusefond } \thanks{L. Decreusefond was partially
  supported by ANR Masterie.}  \thanks{The authors would like to thank
  the anonymous referee for his/her thorough reading which helped us
  to improve the presentation of the current article.}
\email{laurent.decreusefond@telecom-paristech.fr} \author{E. Ferraz}
\email{eduardo.ferraz@telecom-paristech.fr}
\author{H. Randriambololona }
\email{hugues.randriam@telecom-paristech.fr} \author{A. Vergne}
\email{anais.vergne@telecom-paristech.fr} \address{Institut Telecom,
  Telecom ParisTech, CNRS LTCI, 46, rue Barrault, Paris - 75634,
  France} 

\begin{abstract}
  Given a Poisson process on a $d$-dimensional torus, its random
  geometric simplicial complex is the complex whose vertices are the
  points of the Poisson process and simplices are given by the
  \u{C}ech complex associated to the coverage of each point. By means
  of Malliavin calculus, we compute explicitly the three first order
  moments of the number of $k$-simplices, and provide a way to compute
  higher order moments.  Then, we derive the mean and the variance of
  the Euler characteristic.  Using the Stein method, we estimate the
  speed of convergence of the number of occurrences of any connected
  subcomplex converges towards the Gaussian law when the intensity of
  the Poisson point process tends to infinity. We use a concentration
  inequality for Poisson processes to find bounds for the tail
  distribution of the Betti number of first order and the Euler
  characteristic in such simplicial complexes.

\end{abstract}
\maketitle
\section{Motivation}

Algebraic topology is the domain of mathematics in which the
topological properties of a set are analyzed through algebraic tools.
Initially developed for the classification of manifolds, it is by now
heavily used in image processing and geometric data analysis. More
recently, applications to sensor networks were developed
in~\cite{silvacontrol,ghrist}. Imagine that we are given a bounded
domain in the plane and sensors which can detect intruders within a
fixed distance. The so-called coverage problem consists in determining
whether the domain is fully covered, i.e. whether there is any part of
the domain in which an intrusion can occur without being detected. The
mathematical set which is to be analyzed here is the union of the
balls centered on each sensor.  If this set has no hole then the
domain is fully covered. It turns out that algebraic topology provides
a computationally effective procedure to determine whether this
property holds. In view of the rapid development of the technology of
sensor networks \cite{kahn,lewis,pottie}, which are small and cheap
devices with limited capacity of autonomy and communications, devoted
to measure some local physical quantity (temperature, humidity,
intrusion, etc.), this kind of question is likely to become recurrent.

The coverage problem, via homology techniques, for a set of sensors
was first addressed in the papers \cite{silvacontrol,ghrist}. The
method consists in calculating from the geometric data, a
combinatorial object known as a simplicial complex which is a list of
points, edges, triangles, tetrahedron, etc. satisfying some
compatibility conditions: all the faces of a $k$-simplex ($k=0$ means
points, $k=1$ means edges, etc.) of the complex must belong to the set
of $(k-1)$-simplices of the complex. Then, an algebraic structure on
these lists and linear maps, known as boundary operators, are
constructed. Some of the topological properties (like connectivity and
coverage) are given by the so-called Betti numbers which
mathematically speaking are dimension of some quotient vector spaces
(see below). Another key parameter is the alternated sum of the Betti
numbers, known as Euler characteristic which gives some information on
the global topology of the studied set.  Persistence homology
\cite{MR2279866,MR1949898} is both a way to compute the Betti numbers
avoiding a (frequent) combinatorial explosion and a way to detect the
robustness of the topological properties of a set with respect to some
parameter: For instance, in the intrusion detection setting, how the
connectivity of the covering domain is altered by variations of the
detection distance.

When points (i.e. sensors) are randomly located in the ambient space,
may it be $\R^d$ or a manifold, it is natural to ask about the
statistical properties of the Betti numbers and the Euler
characteristic.  We completely solved the problem in one dimension
(see \cite{DF:JPS-11}) by basic methods inspired by queuing theory,
without using the forthcoming sophisticated tools of algebraic
topology. Since we cannot order points in $\R^d$, it is not possible
to generalize the results obtained in this earlier work to higher
dimension.

A very few papers deal with the properties of random simplicial
complexes. In~\cite{MR2770552} and \cite{MR1986198}, for Binomial
point processes whose number of points are going to infinity, the
asymptotic regimes of the mean value of the Betti numbers and
simplices numbers are investigated. In \cite{Kahle2010}, these results
are refined by providing Poisson and Gaussian approximations of the
Betti numbers in asymptotic regimes.

As will be apparent below, for the Rips-Vietoris simplicial complex,
the number of $k$-simplices boils down to the number of
$(k+1)$-cliques of the underlying graph. As we mainly analyze this
kind of complex, our work has thus strong links with the pioneering
work of Penrose \cite{MR1986198} and with \cite{MR1193616,Penrose2011}
as well. In \cite{MR1193616}, for Poisson input, the limiting regimes
of the number of $k$-simplices on a square of size $a$, are
investigated through limit theorems of U-statistics. The size of the
square is growing to infinity with a constant mean number of points
per unit of surface.  In \cite{Penrose2011}, there is an extension of
the latter result to non-linear manifolds and non-uniform
distributions of the points.
In the above cited papers, sophisticated combinatorial arguments are
at the root of the arguments of the convergence theorems. We here
replace this line of thought by a functional analytic approach which
transfers the difficulty to the computation of a (possibly involved)
deterministic integral. By doing so, we can, in principle, obtain a
CLT for the number of occurrences of any connected sub-complex and not
only for cliques.

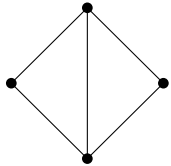
\begin{figure}[!ht]
  \label{fig_chi_mean_v2:not_clique}
  \centering
  \begin{tikzpicture}
    \draw (0,0) -- (1,1) -- (2,0) -- (1,-1) -- cycle; \draw (1,1) --
    (1,-1); \fill (0,0) circle(2pt); \fill (1,1) circle(2pt); \fill
    (1,-1) circle(2pt); \fill (2,0) circle(2pt);
  \end{tikzpicture}
  \caption{A sub-complex which is not a clique}
\end{figure}

One of our main contributions are exact formulas for the first three
moments of the number of simplices for both the Poisson and the
Binomial processes at the price of working on a square bounded domain,
which we embed into a torus in order to avoid side effects. The
rationale behind this simplification is that when the size of the
covering balls is small compared to the size of the square, the
topology of the two sets (the union of balls in the square and the
union of balls in the corresponding torus) must be similar. Our method
could be generalized to compute the moments of any order but the
computations become more and more tricky as the order increases. We
also investigate the properties of the moments of the Euler
characteristic.  Moreover, by using Malliavin calculus, we go further
than the previously cited works since we can evaluate the speed of
convergence in the CLT. We also give a concentration inequality to
bound the distribution tail of the first Betti number.

During the final preparation of this paper, we learned that such an
approach was used independently in \cite{Reitzner2011} for the
analysis of $U$-statistics (functionals of a fixed chaos in our
vocabulary) of Poisson processes. Both approaches rely on the ideas
which appeared in \cite{Decreusefond:2008jc,Nourdin:2012fk}.

Our method goes as follows: We write the numbers of $k$-simplices
(i.e. points, edges, triangles, tetrahedron, etc) as iterated
integrals with respect to the underlying Poisson process. Then, the
computation of the means is reduced to the computation of
deterministic iterated integrals thanks to Campbell formula. By using
the definition of the Euler characteristic as an alternating sum of
the numbers of simplices, we find its expectation. The point is that
even if the summing index goes to infinity, the expectation of $\chi$
depends only on the $d$-th power of the intensity of the Poisson
process where $d$ is the dimension of the underlying space.  By
depoissonization, we obtain the exact values of the mean number of
simplices of any order and then the mean value of the Euler
characteristic for Binomial processes.  Using the multiplication
formula of iterated integrals, one can reproduce the same line of
thought for higher order moments to the price of an increased
complexity in the computations. We obtain closed form formulas for the
variance of the number of $k$-simplices and of the Euler
characteristic and a series expansion for third order moments. Our
method is a priori suitable for any higher order moments but the
computations become much involved. Using Stein's method mixed with
Malliavin calculus, we generalize the results of~\cite{MR1986198} by
proving a precise (i.e. with speed of convergence) CLT for
sub-complexes count. As expected, the speed of convergence is of the
order of $\lambda^{-1/2}$.

The paper is organized as follows: Sections 2 and 3 are primers on
algebraic topology and Malliavin calculus. In Section 4, the average
number of simplices and the mean of the Euler characteristic are
computed. This is sufficient to bound the tail distribution of
$\beta_0$ using concentration inequality.  Section 5 applies the
Malliavin calculus in order to find the explicit expression of
second order moments of the number of $k$-simplices and the Euler
characteristic. Using the same strategy, in Section 6, we find the
expression for the third order moment of the number of simplices.  In
Section 7, we prove a central limit theorem for the number of
occurrences of a finite simplex into a Poisson random geometric
complex.

\section{Algebraic Topology}
\label{sec:algebraic-topology}
For further reading on algebraic topology,
see~\cite{armstrong,hatcher,munkres}. Graphs can be generalized to
more generic combinatorial objects known as simplicial
complexes. While graphs model binary relations, simplicial complexes
represent higher order relations. Given a set of vertices $V$, a
$k$-simplex is an unordered subset $\{v_0,\,v_1,\,\dots,\, v_k\}$
where $v_i\in V$ and $v_i\not=v_j$ for all $i\not=j$. The faces of the
$k$-simplex $\{v_0,\, v_1,\, \dots,\, v_k\}$ are defined as all the
$(k-1)$-simplices of the form $\{v_0,\,\dots,\,v_{j-1},\,v_{j+1},\,
\dots,\, v_k\}$ with $0\leq j\leq k$. A simplicial complex $\mathcal
C$ is a
collection of simplices which is closed with respect to the inclusion
of faces, i.e. if $\{v_0,\, v_1,\, \dots,\, v_k\}$ is a $k$-simplex
then all its faces are in the set of $(k-1)$-simplices.

One can define an orientation on simplices by defining an order on
vertices and with the convention that:
\begin{eqnarray*}
  [v_0,\, \dots,\, v_i,\, \dots,\, v_j,\, \dots,\, v_k]=-[\, v_0,\,
  \dots,\, v_j,\, \dots,\, v_i,\,\dots,\, v_k],
\end{eqnarray*}
for $0 \leq i,j \leq k$.

For each integer $k$, $\mathcal C_k$ is the vector space spanned by the set
of oriented $k$-simplices of ${\mathcal V}$. For any integer $k$, the boundary
map $\partial_k$ is the linear transformation $\partial_k\, :\,
\mathcal C_k\rightarrow \mathcal C_{k-1}$ which acts on basis elements
$[v_0,\dots,v_k]$ as
\begin{eqnarray*}
  \label{eq:erased}
  \partial_k [v_0,\,\dots,\, v_k]= \sum_{i=0}^{k}(-1)^i
  [v_0,\,\dots,\, v_{i-1},\, v_{i+1},\, \dots,\,v_k],
\end{eqnarray*}
and $\partial_0$ is the null function. Examples of such operations are
given in Table~\ref{fig. boundary}.

\begin{table}[!ht]
  \begin{tabular}{ccc}
    \begin{tikzpicture}[scale=0.5, font=\tiny]
      \draw[color=black] (0,0) -- (1,1) -- (2,0); \draw[-stealth]
      (0,0) -- (.5,.5); \draw[-stealth] (1,1) -- (1.5,.5); \filldraw
      [color=black] (0,0) node[below] {$v_0$} circle (.1); \filldraw
      [color=black] (1,1) node[above] {$v_1$} circle (.1); \filldraw
      [color=black] (2,0) node[below] {$v_2$} circle (.1); \filldraw
      [color=black] (3,0) node[below] {$v_0$} node[above] {$-$}
      circle (.1); \filldraw [color=black] (5,0) node[below] {$v_2$}
      node[above] {$+$ } circle (.1); \node at (2,-2)
      {$[v_0,v_1]+[v_1,v_2]\stackrel{\partial}{\longrightarrow}
      [v_2]-[v_0]$};
    \end{tikzpicture}
    &
    \begin{tikzpicture}[scale=0.5, font=\tiny]
      \filldraw[color=blue!50, draw=black] (0,0) -- (1,1) -- (2,0)--
      cycle; \draw[draw=black] (3,0) -- (4,1) -- (5,0)-- cycle;
      \draw[->] (1,.2) arc (270:-30:.2); \draw[-stealth] (0,0) --
      (.5,.5); \draw[-stealth] (1,1) -- (1.5,.5); \draw[-stealth]
      (2,0) -- (1,0); \draw[-stealth] (3,0) -- (3.5,.5);
      \draw[-stealth] (4,1) -- (4.5,.5); \draw[-stealth] (5,0) --
      (4,0); \filldraw [color=black] (0,0) node[below] {$v_0$}
      circle (.1); \filldraw [color=black] (1,1) node[above] {$v_1$}
      circle (.1); \filldraw [color=black] (2,0) node[below] {$v_2$}
      circle (.1); \filldraw [color=black] (3,0) node[below] {$v_0$}
      circle (.1); \filldraw [color=black] (4,1) node[above] {$v_1$}
      circle (.1); \filldraw [color=black] (5,0) node[below] {$v_2$}
      circle (.1); \node at (2,-2)
      {$[v_0,v_1,v_2]\stackrel{\partial}{\longrightarrow}
      [v_1,v_2]-[v_0,v_2]$}; \node at (3.3,-2.7) {$+[v_0,v_1]$};
    \end{tikzpicture}
    &
    \begin{tikzpicture}[scale=0.5, font=\tiny]
      \filldraw[color=blue!40, draw=black] (0,0) -- (0,1.5) --
      (-1.5,.75)-- cycle; \filldraw[color=blue!60, draw=black] (0,0)
      -- (0,1.5) -- (1,.75)-- cycle; \draw[->] (-.6,1) arc
      (110:-150:.15 and .3); \draw[->] (.4,.5) arc (235:-25:.1 and
      .3); \filldraw [color=black] (0,0) node[below] {$v_0$} circle
      (.1); \filldraw [color=black] (0,1.5) node[above] {$v_1$}
      circle (.1); \filldraw [color=black] (-1.5,.75) node[left]
      {$v_2$} circle (.1); \filldraw [color=black] (1,.75)
      node[below] {$v_3$} circle (.1);

      \draw[dashed] (2.5,.75) -- (5,.75); \fill[color=blue!40,
      opacity=.4] (5,.75) -- (4,1.5) -- (2.5,.75)-- cycle;
      \filldraw[color=blue!40, draw=black, opacity=.4] (4,0) --
      (4,1.5) -- (2.5,.75)-- cycle; \filldraw[color=blue!80,
      draw=black, opacity=.4] (4,0) -- (4,1.5) -- (5,.75)-- cycle;

      \node at (-.1,-.8) {Filled}; \node at (3.9,-.8) {Empty};

      \draw[->, densely dotted] (4.1,.85) arc (-40:220:.25 and .15);
      \draw[->, densely dotted] (4.2,.3) arc (-40:220:.3 and .15);
      \draw[->] (3.4,1) arc (110:-150:.15 and .3); \draw[->]
      (4.4,.5) arc (235:-25:.1 and .3);

      \filldraw [color=black] (4,0) node[below] {$v_0$} circle (.1);
      \filldraw [color=black] (4,1.5) node[above] {$v_1$} circle
      (.1); \filldraw [color=black] (2.5,.75) node[below] {$v_2$}
      circle (.1); \filldraw [color=black] (5,.75) node[right]
      {$v_3$} circle (.1);

      \node[anchor=east] at (2.5,-3.05)
      {$[v_0,v_1,v_2,v_3]\stackrel{\partial}{\longrightarrow}$};
      \node at (4,-2) {$ +[v_1,v_2,v_3]$}; \node at (4,-2.7)
      {$-[v_0,v_2,v_3]$}; \node at (4,-3.4) {$+[v_0,v_1,v_3]$};
      \node at (4,-4.1) {$-[v_0,v_1,v_2]$};

    \end{tikzpicture}\\
    a) & b) & c)
  \end{tabular}
  \caption{Examples of boundary maps. From left to right. An
    application over 1-simplices. Over a 2-simplex. Over a 3-simplex,
    turning a filled tetrahedron to an empty one.}
  \label{fig. beta}
  \label{fig. boundary}
\end{table}
These maps give rise to a chain complex: a sequence of vector spaces
and linear transformations:
\begin{eqnarray*}
  \dots \stackrel{\partial_{k+2}}{\longrightarrow}\mathcal C_{k+1}
  \stackrel{\partial_{k+1}}{\longrightarrow}\mathcal C_{k}
  \stackrel{\partial_{k}}{\longrightarrow}\dots
  \stackrel{\partial_2}{\longrightarrow}\mathcal C_{1}
  \stackrel{\partial_1}{\longrightarrow}\mathcal C_{0}.
\end{eqnarray*}
A standard result then asserts that for any integer $k$,
$\partial_k\circ\partial_{k+1}=0.$ If one defines
$Z_k=\ker \partial_k \text{ and } B_k=\mathrm{im}
\, \partial_{k+1},$ this induces that $B_k\subset Z_k.$

\begin{figure}[!h]
  \centering
  \begin{tikzpicture}[font=\small]
    \draw (0,0) circle (1 and 1.5); \draw (-3,0) circle (1 and 1.5);
    \draw (3,0) circle (1 and 1.5); \filldraw[color=blue!60,
    draw=black] (0,0) circle (.8 and 1.2); \filldraw[color=white,
    draw=black] (0,0) circle (.5 and .75); \fill[color=black] (0,0)
    node[below] {0} circle (0.05); \fill[color=black] (3,0)
    node[below] {0} circle (0.05); \fill[color=black] (-3,0)
    node[below] {0} circle (0.05); \node[above] at (80:1 and 1.5)
    {$\mathcal C_k$}; \node[below] at (90:.8 and 1.2) {$Z_k$};
    \node[below] at (90:.5 and .75) {$B_k$}; \node[above] at
    ($(-3,0)+(80:1 and 1.5)$) {$\mathcal C_{k+1}$}; \node[above] at
    ($(3,0)+(80:1 and 1.5)$) {$\mathcal C_{k-1}$}; \draw ($(-3,0)+(80:1 and
    1.5)$) -- ($(90:.5 and .75)$); \draw ($(-3,0)+(-80:1 and 1.5)$) --
    ($(-90:.5 and .75)$); \draw ($(80:.8 and 1.2)$) -- ($(3,0)$);
    \draw ($(-80:.8 and 1.2)$) -- ($(3,0)$); \draw ($(80:.5 and .75)$)
    -- ($(3,0)$); \draw ($(-80:.5 and .75)$) -- ($(3,0)$); \node at
    (-1.5,0) {$\stackrel{\partial}{\longrightarrow}$}; \node at
    (1.5,0) {$\stackrel{\partial}{\longrightarrow}$};
  \end{tikzpicture}
  \label{fig: sets}
  \caption{A chain complex showing the sets $\mathcal C_k$, $Z_k$ and
    $B_k$.}
\end{figure}
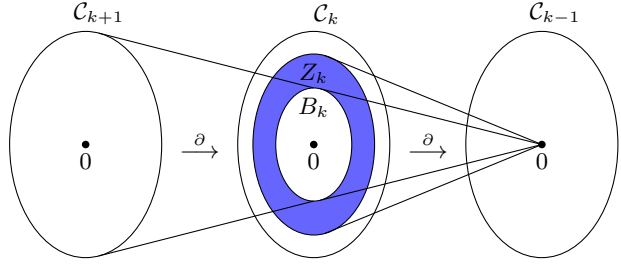

The $k$-th homology group of ${\mathcal C}$, denoted by $H_k$, is the quotient
vector space,
\begin{math}
  \label{eq:erased2}
  H_k={Z_k}/{B_k}
\end{math}
and the $k$-th Betti number of ${\mathcal C}$ is its dimension:
$\beta_k=\dim H_k=\dim Z_k-\dim B_k.$

The simplicial complexes we consider are of a special type since they
are built on topological rules.
\begin{definition}
  Given $\mathcal{U}=(U_v,\, v\in {\omega})$ a collection of open
  sets of some topological space $X$, the \u{C}ech complex of $\mathcal{U}$ denoted by
  $\mathcal{C}(\mathcal{U}),$ is the abstract simplicial complex whose
  $k$-simplices correspond to $(k+1)$-tuples of distinct elements of
  $\mathcal{U}$ that have non empty intersection, so $\{v_0,\,v_1,\,
  \dots,\,v_k\}$ is a $k$-simplex if and only if $\bigcap_{i=0}^k
  {U_{v_i}}\not=\emptyset$.
\end{definition}
For the applications we have in mind, the set $U_v$ is meant to be the covered zone by the sensor
located at $v$. For the sake of tractability, it is supposed to be
a ball centered at $v$ with a fixed radius. As said earlier, in order to avoid side effects, we work on the
$d$-dimensional torus of length $a$, denoted by $\mathbb{T}_a^d$ and
to simplify the computations, we consider the $l^\infty$
distance. Namely, the torus is defined as the quotient of the
action of the group of translations $a{\mathbb Z}^d$ on $\R^d$,
i.e. $\mathbb{T}_a^d=\R^d/a{\mathbb Z}^d$. The space $X=[0,\, a)^d$
can be embedded in $\mathbb{T}_a^d$ as a fundamental domain of this
action. If we equip $X$ with the distance
\begin{equation*}
  {\rho_d}(x,\, y)=\inf_{k\in {\mathbb Z}^d} \| x-y+ka\|_\infty
\end{equation*}
where $\|x\|_\infty$ is the $l^\infty$-norm in $\R^d$, then the
embedding of $X$ into the torus is a bijective isometry. One can thus
identify these two spaces and use the representation which is the most
convenient according to the situation.
\begin{definition}
  Given $\omega$ a finite set of points on the torus
  $\mathbb{T}_a^d$. For $\epsilon >0$, we define
  $\mathcal{U}_\epsilon(\omega)=\{B_{\rho_d}(v,\, \epsilon),\, v\in
  \omega\}$ and
  $\mathcal{C}_{\epsilon}(\omega)=\mathcal{C}(\mathcal{U}_\epsilon(\omega))$,
  where $B_{\rho_d}(x,\, \epsilon)=\{y\in \mathbb{T}_a^d,\, {\rho_d}(x,\, y) < \epsilon\}$.
 \end{definition}
The following result is known for $\R^d$, there is a slight
modification of the proof for the torus.
\begin{theorem}
  \label{th: cech}
  Suppose $\epsilon<a/4$. Then $\mathcal{C}_{\epsilon}(\omega)$ has
  the same homology vector spaces as $\U_\epsilon(\omega)$. In
  particular they have the same Betti numbers.
\end{theorem}
\begin{proof}
  This will follow from the so-called nerve lemma of Leray, as stated
  in \cite[Theorem 7.26]{rotman} or \cite[Theorem 10.7]{bjorner}.  One
  only needs to check that any non-empty intersection of sets
  $B_{\rho_d}(v,\epsilon)$ is contractible.

  Consider such a non-empty intersection, and let $x$ be a point
  contained in it. Then, since $\epsilon<a/4$, the ball
  $B_{\rho_d}(x,\, 4\epsilon)$ can be identified with a cube in the
  Euclidean space. Then each $B_{\rho_d}(v,\, \epsilon)$ containing
  $x$ is contained in $B_{\rho_d}(x,4\epsilon)$, hence also becomes
  a cube with this identification, hence convex. Then the intersection
  of these convex sets is convex, hence contractible.
\end{proof}
For any finite sets of points $\omega$ of the torus, according to the
geometrical definition of the \u{C}ech complex, the Betti numbers have a
geometrical meaning:  $\beta_0(\omega)$ (with obvious notations) is
the number of connected components and for $k\ge 1$, $\beta_k(\omega)$ is
the number of $k$-dimensional holes of~$\mathcal{U}_\epsilon(\omega)$. For $k=0$ and $k=1$, an
intuitive explanation can be given. By definition, $\beta_0(\omega)$ is the
number of points minus the number of independent edges. Each time
there exists a cycle with $n$ points, we can remove an edge without
altering $\beta_0(\omega)$ since there are $n-1$ independent edges in such
a cycle. Doing this repeatidly, one can reduce the original graph to
as many linear chains of edges as there are connected components. A
linear chain of edges which contains $n$ points has $n-1$ edges, hence
a $\beta_0(\omega)$ equal to $1$. Thus $\beta_0(\omega)$ counts the number of
connected components. As to $\beta_1(\omega)$, we remark that
$\ker \partial_1$ is composed by the cycles and that $B_1$ is the
set of linear combinations of edges forming triangles, hence $\beta_1(\omega)$
is the number of cycles which are not triangles, hence  represents the number of ``coverage'' holes. 
The well known topological invariant named Euler characteristic for
${\mathcal U}_\epsilon(\omega)$, denoted by $\chi(\omega)$, is an integer defined by:
\begin{eqnarray*}
  \label{eq: chidef1}
  \chi(\omega)=\sum_{i=0}^{\infty}(-1)^i\beta_i(\omega).
\end{eqnarray*}
Let $s_k(\omega)$ be the number of $k$-simplices in the simplicial complex
$\mathcal C_\epsilon(\omega)$. A well known theorem states that this is also given by:
\begin{eqnarray*}
  \label{eq: chidef2}
  \chi(\omega)=\sum_{i=0}^{\infty}(-1)^is_i(\omega).
\end{eqnarray*}
\begin{definition}
  Let $\omega$ be a finite set of points in ${\mathbb T}_a^d$.  For
  any $\epsilon>0$, the Rips-Vietoris complex of $\omega$,
  $\mathcal{R}_{\epsilon}(\omega)$, is the abstract simplicial complex whose
  $k$-simplices correspond to unordered $(k+1)$-tuples of points in
  $\omega$ which are pairwise within distance less than $\epsilon$ of
  each other.
\end{definition}
\begin{lemma}
  \label{lem: rips-cech}
  For the torus $\mathbb{T}^d_{a}$ equipped with the product distance
  $\rho_d$, the Rips-Vietoris complex
  $\mathcal{R}_{2\epsilon}(\omega)$ has the same Betti numbers as the
  \u{C}ech complex $\mathcal{C}_{\epsilon}(\omega)$.
\end{lemma}
The proof is given in~\cite{ghrist} in a slightly different context,
but it is easy to check that it works here as well. It must be pointed
out that \u{C}ech and Rips-Vietoris simplicial complexes can be
defined similarly for any distance on $\mathbb{T}^d_{a}$ but it is
only for the product distance that the homology vector spaces of both
complexes coincide.

\begin{proposition}
  \label{prop: beta_d}
  Let $\omega\in \mathbb{T}^d_{a}$ be a set of points, generating the
  simplicial complex $\mathcal{C}_{\epsilon}(\omega)$. Then, if $i>d$,
  $\beta_i(\omega)=0$.
\end{proposition}
\begin{proof}
  By Theorem~\ref{th: cech}, $\mathcal{C}_{\epsilon}(\omega)$ has the
  same homology as ${\mathcal U}_\epsilon(\omega)$.  But ${\mathcal
    U}_\epsilon(\omega)$ is an open manifold of dimension $d$, so its
  Betti numbers $\beta_i(\omega)$ vanish for $i>d$, see for example
  \cite[Theorem 22.24]{greenbergharper}.
\end{proof}

\begin{proposition}
  \label{prop:chi_interpretation}
  Let $\omega\in \mathbb{T}^d_{a}$ be a set of points, generating the
  simplicial complex $\mathcal{C}_{\epsilon}(\omega)$. There are only
  two possible values for the $d$-th Betti number of
  $\mathcal{C}_{\epsilon}(\omega)$:
  \begin{enumerate}[i)]
  \item $\beta_d(\omega)=0$, or
  \item $\beta_d(\omega)=1$.
  \end{enumerate}
  If the latter condition holds, then we also have
  $\chi(\omega)=0$.
\end{proposition}
\begin{proof}
  By Theorem~\ref{th: cech}, $\mathcal{C}_{\epsilon}(\omega)$ has the
  same homology as ${\mathcal U}_\epsilon(\omega)$.  Now, ${\mathcal
    U}_\epsilon(\omega)$ is an open sub-manifold of the torus, so
  there are only two possibilities:
  \begin{enumerate}[i)]
  \item ${\mathcal U}_\epsilon(\omega)$ is a strict open sub-manifold,
    hence non-compact
  \item ${\mathcal U}_\epsilon(\omega)=\mathbb{T}^d_{a}$.
  \end{enumerate}
  In the first case, $\beta_d(\omega)=0$ by \cite[Corollary
  22.25]{greenbergharper}.  In the second case
  $\mathcal{C}_{\epsilon}(\omega)$ has same homology as the torus,
  hence $\beta_d(\omega)=1$ and $\chi(\omega)=0$.
\end{proof}

\section{Poisson point process and Malliavin calculus}
The space of configurations on $X=[0,\, a)^d$, is the set of locally
finite simple point measures (see \cite{verejones,MR2531026} for
details):
\begin{equation*}
  \label{eq: def_poisson}
  \Omega^{X}=\left\{\omega=\sum_{k=1}^n\delta(x_k)\ :\ (x_k)_{k=1}^{k=n} \subset X,\ n\in\mathbb{N}\cup\{\infty\}\right\},
\end{equation*}
where $\delta(x)$ denotes the Dirac measure at $x\in X$.  It is
often convenient to identify an element $\omega$ of $\Omega^X$ with
the set corresponding to its support, i.e. $\sum_{k=1}^n\delta(x_k)$
is identified with the unordered set $\{x_1,\, \dots,\, x_n\}$.  For
$A\in \mathcal{B}(X)$, we have $\delta(x)(A)=\indicator{A}(x)$, so that
\begin{eqnarray*}
  \omega(A)= \sum_{x\in\omega}\indicator{A}(x),
\end{eqnarray*}
counts the number of atoms in $A$. Simple
measure means that $\omega(\{x\})\le 1$ for any $x\in X$. Locally
finite means that $\omega(K)<\infty$ for any compact $K$ of $X$. The configuration space
$\Omega^{X}$ is endowed with the vague topology and its associated
$\sigma$-algebra denoted by $\mathcal{F}^{X}$.
To characterize the randomness of the system, we consider that the set
of points is represented by a Poisson point process $\omega$ with
intensity measure $\d\Lambda(x)=\lambda \d x$ in $X$. The
parameter $\lambda$ is called the intensity of the Poisson
process.  Since $\omega$ is a
Poisson point process of intensity measure $\Lambda$:
\begin{enumerate}[i)]
\item For any compact $A$, $\omega(A)$ is a random variable of
  parameter $\Lambda(A)$:
  \begin{equation*}
    \Pr(\omega(A)=k)=e^{-\Lambda(A)}\frac{\Lambda(A)^{k}}{k!}\cdotp
  \end{equation*}
\item For any disjoint sets  $A, A'\in \mathcal{B}(X)$, the random
  variables $\omega(A)$ and $\omega(A')$ are independent.
\end{enumerate}

Along this paper, we refer $\E{F}$ as the mean of some function $F$
depending on $\omega$ given that the intensity measure of this process is
$\Lambda$. The notations $\var{F}$ and $\cov{F,\,G}$ are defined accordingly. As said above, a configuration $\omega$ can be viewed
as a measure on $X$. It also induces a measure on any $X^n$, called
the factorial measure associated to $\omega$ of order $n$, defined by
\begin{equation*}
  \omega^{(n)}(C)=\sum_{\substack{(x_1,\cdots,\, x_n)\in \omega\\
      x_i\not = x_j}}\indicator{C}(x_1,\,\cdots,\, x_n),
\end{equation*}
for any $C\in X^n$, with the convention that $\omega^{(n)}$ is the
null measure if $\omega$ has less than $n$ points.  Let $f\in
L^1(\Lambda^{\otimes n})$ and let $F$ be
a random variable given by
\begin{equation*}
  F(\omega)=\sum_{\substack{x_i\in
      \omega\\x_i\not=x_j}}f(x_1,\,\dots,\,x_n)=\int f(x_1,\dots,x_n)\d\omega^{(n)}(x_1,\,\cdots,\,
  x_n).
\end{equation*}
The Campbell-Mecke formula for Poisson point processes states that
\begin{equation*}
  \E{F}=\int_{X^n} f(x_1,\,\dots,\, x_n)\d\Lambda(x_1)\dots\d\Lambda(x_n).
\end{equation*}
In view of this result, it is natural to introduce the compensated
factorial measures defined by :
\begin{equation*}
  d\omega^{(1)}_\Lambda(x)=d\omega(x)-\d\Lambda(x)
\end{equation*}
and for $n\ge 2$, for any $f\in L^1(\Lambda^{\otimes n})$,
\begin{multline*}
  \int f(x_1,\, \cdots,\, x_n) d\omega_\Lambda^{(n)}(x_1,\,
  \cdots,\, x_n)\\= \int \left(\int f(x_1,\, \cdots,\,
    x_n) (d(\omega-\sum_{j=1}^{n-1}\delta(x_j))(x_n)-\d\Lambda(x_n))
  \right)\\ d\omega_\Lambda^{(n-1)}(x_1,\, \cdots,\, x_{n-1}).
\end{multline*}
A real-valued function $f\ :\ X^n \rightarrow \mathbb{R}$ is called symmetric
if
\begin{eqnarray*}
  f(x_{\sigma(1)},\, \dots,\, x_{\sigma(n)})=f(x_1,\, \dots,\, x_n)
\end{eqnarray*}
for all permutations $\sigma$ of $\SS_n$.
Then the space of square integrable symmetric functions of $n$
variables is denoted by $L^2(X,\Lambda)^{\circ n}$. For $f\in
L^2(X,\Lambda)^{\circ n}$, the multiple Poisson stochastic
integral $I_n(f_n)$ is then defined as
\begin{equation*}
  I_n(f_n)(\omega)=\int f_n(x_1,\dots,x_n) \d\omega_\Lambda^{(n)}(x_1,\,
  \cdots,\, x_n).
\end{equation*}
It is known that $I_n(f_n) \in L^2(\Omega^X, \P)$.  Moreover, if
$f_n\in L^2(X,\,\Lambda)^{\circ n}$ and $g_m\in L^2(X,\,\Lambda)^{\circ m}$, the isometry formula
\begin{equation}
  \label{eq: isometry}
\E{I_n(f_n)I_m(g_m)}
  =n!\,\indicator{m}(n) \ 
  \langle  f_n,g_m\rangle_{L^2(X,\Lambda)^{\circ n}} 
\end{equation}
holds true.  Furthermore, we have:
\begin{theorem}
  Every random variable $F\in L^2(\Omega^{X},\, \P)$ admits a unique
  Wiener-Poisson decomposition of the type
  \begin{equation*}
    F=\esp{F}+\sum_{n=1}^{\infty}I_n(f_n),
  \end{equation*}
  where the series converges in $L^2(\Omega^X,\P)$ and, for each
  $n\geq1$, the kernel $f_n$ is an element of $L^2(X,\Lambda)^{\circ n}$. Moreover, by definition $\var{F}=\Vert
  F-\esp{F}\Vert_{L^2(\Omega^X,\P)}^2$ then we have the isometry
  \begin{equation}
    \label{eq: isometry2}
    \var{F}=\sum_{n=1}^{\infty}n!\ \Vert f_n\Vert_{L^2(X,\Lambda) ^{\circ n}}^2.
  \end{equation}
\end{theorem}
For $f_n\in L^2(X,\Lambda)^{\circ n}$ and $g_m\in L^2(X,\Lambda)^{\circ m}$, we define $f_n\otimes_{k}^l g_m$, $0\leq l\leq k$,
to be the function:
\begin{multline}
  \label{eq:contraction}
  (y_{l+1},\dots,y_n,x_{k+1},\dots,x_m)\longmapsto \\
  \int_{X^l} f_n(y_1,\dots,y_n)g_m(y_1,\dots,y_k,x_{k+1},\dots,x_m)\d
  \Lambda(y_1)\ldots\d \Lambda(y_l).
\end{multline}
We denote by $f_n\circ_{k}^l g_m$ the symmetrization in $n+m-k-l$
variables of $f_n\otimes_{k}^l g_m$, $0\leq l\leq k$. This leads us to
the next proposition (see \cite{MR2531026} for a proof):
\begin{proposition}
  \label{prop:chaos_prod}
  For $f_n\in L^2(X,\Lambda)^{\circ n}$ and $g_m\in L^2(X,\Lambda)^{\circ m}$, we have
  \begin{equation*}
    I_n(f_n)I_m(g_m)=\sum_{s=0}^{2(n\wedge m)}I_{n+m-s}(h_{n,m,s}),
  \end{equation*}
  where
  \begin{equation*}
    h_{n,m,s}=\sum_{s\leq 2i\leq 2(s\wedge n\wedge m)}i!{n\choose i}{m\choose i}{i\choose s-i}f_n\circ_{i}^{s-i} g_m
  \end{equation*}
  belongs to $L^2(X,\Lambda)^{\circ n+m-s}$, $0\leq s\leq
  2(m\wedge n)$.
\end{proposition}
In what follows, given $f\in L^2(X,\Lambda)^{\circ q}$ $(q\geq2)$
and $t\in X$, we denote by $f(*,x)$ the function on $X^{q-1}$ given by
$(x_1,\dots,x_{q-1})\longmapsto f(x_1,\dots,x_{q-1},x)$.
\begin{definition}
  \label{eq: def malliavin}
  Let $\dom D$ be the the set of random variables $F\in
  L^2(\Omega^X,\, \P)$ admitting a chaotic decomposition such that
  \begin{equation*}
    \sum_{n=1}^{\infty}qq!\Vert f_n\Vert_{L^2(X,\Lambda)^{\circ n}}^2<\infty.
  \end{equation*}
  Let $D$ be defined by
  \begin{align*}
    D\ :\ \dom D&\rightarrow L^2(\Omega^{X}\times X,\P\otimes \Lambda)\\
    F=\esp{F}+\sum_{n\ge 1} I_n(f_n)& \longmapsto D_xF=\sum_{n\ge 1}
    nI_{n-1}(f_n(*,x)).
  \end{align*}
  It is known, cf.~\cite{ito}, that we also have
  \begin{equation*}
    \label{eq:ito}
    D_xF(\omega)=F(\omega\cup\{x\})-F(\omega),\ \P\otimes \Lambda-a.e.
  \end{equation*}
\end{definition}
\begin{definition}
  The Ornstein-Uhlenbeck operator $L$ is given by
  \begin{equation*}
    LF=-\sum_{n=1}^{\infty}nI_{n}(f_n),
  \end{equation*}
  whenever $F\in \dom L$, given by those $F\in L^2(\Omega^X,\, \P)$ such
  that their chaos expansion verifies
  \begin{equation*}
    \sum_{n=1}^{\infty}q^2q!\Vert f_n\Vert_{L^2(X,\Lambda)^{\circ n}}^2<\infty.
  \end{equation*}
  Note that $\E{LF}=0$, by definition and \eqref{eq: isometry}.
\end{definition}
\begin{definition}
  For $F\in L^2(\Omega^X,\P)$ such that $\esp{F}=0$, we may define
  $L^{-1}$ by
  \begin{equation*}
    L^{-1}F=-\sum_{n=1}^{\infty}\frac{1}{n}I_{n}(f_n).
  \end{equation*}
\end{definition}
Combining Stein's method and Malliavin calculus yields the following
theorem, see~\cite{taqqu}:
\begin{theorem}
  \label{thm_chi_mean_v2:1}
  Let $F\in\dom D$ be such that $\E{F}=0$ and Var$(F)=1$. Then,
  \begin{multline*}
    \label{eq: taqqu}
    d_W(F,\, \mathcal{N}(0,1))\leq \E{\left| 1+\int_{X}[ D_xF\times
        D_xL^{-1}F]\d\Lambda(x)\right|}\\
    +\int_{X}\E{\left| D_xF\right|^2 \left|
        D_xL^{-1}F\right|}\d\Lambda(x).
  \end{multline*}
\end{theorem}
Another result from the Malliavin calculus used in this work is the
following one, quoted from \cite{MR2531026}:
\begin{theorem}
  \label{th:conc_inequal}
  Let $F\in \dom D$ be such that $DF\leq K$, a.s., for some $K\geq0$
  and we denote
$$\Vert DF\Vert_{L^{\infty}(L^{2}(X,\Lambda),\P)} : =\sup_\omega
\int_X |D_x F(\omega)|^2 \d \Lambda(x)< \infty.$$ Then
\begin{equation}
  \label{eq: conc_inequal}
  \Pr(F-\E{F}\geq x)\leq \exp\left(-\frac{x}{2K}\log\left(1+\frac{xK}{\Vert DF\Vert_{L^{\infty}(L^{2}(X,\,\Lambda),\P)}  } \right) \right).
\end{equation}
\end{theorem}

\section{First order moments}

Let $\omega$ denote a generic realization of a Poisson point process
on the torus $\mathbb{T}_a^d$ and $\mathcal C_{\epsilon}(\omega)$ the
associated 
\u{C}ech complex with $\epsilon<a/4$. A Poisson process in
$\R^d$ of intensity $\lambda$ dilated by a factor $\alpha$ is a
Poisson process of intensity $\lambda \alpha^{-d}$. Hence,
statistically, the homological properties of a Poisson process of
intensity $\lambda$, inside a torus of length $a$ with ball sizes
$\epsilon$ are the same as that of a Poisson process of intensity
$\lambda \alpha^{-d}$, inside a torus of length $\alpha a$ with ball
sizes $\alpha \epsilon$.  Thus there are only two degrees of freedom
among $\lambda$, $a$, and $\epsilon$. For instance, we can set $a=1$
and the general results are obtained by a multiplication of magnitude
$a^d$. Strictly speaking, Betti numbers, Euler characteristic and
number of $k$-simplices are functions of $\mathcal
C_{\epsilon}(\omega)$ but we will skip this dependence for the sake of
notations. We also define $N_k$ as the number of $(k-1)$-simplices.

In this section, we evaluate the mean of the number of
$(k-1)$-simplices $\E{N_k}$ and the mean of the Euler characteristic
$\E{\chi}$.  We introduce some notations.  Let
\begin{equation*}
  \Delta_{k}^{(d)}=\{(v_1,\dots,v_k) \in ([0,\, a)^d)^k,\  v_i \neq v_j,
  \forall i\neq j \}.
\end{equation*}
For any integer $k$, we define $\varphi_k^{(d)}$ as:
\begin{align*}
  \varphi_k^{(d)} : ([0,\, a)^d)^k & \longrightarrow \{0,1\}\\
  (v_1,\,\cdots,\, v_k) & \longmapsto 
  \begin{cases}
\prod_{1\le i<j\le k}{\mathbf
    1}_{[0,\, 2\epsilon)}({\rho_d}(v_i,\, v_j)) & \text{ if } (v_1,\dots,v_k) \in  \Delta_{k}^{(d)},\\    
 0 & \text{ otherwise.}
  \end{cases}
\end{align*}
In words, this means that $ \varphi_k^{(d)}(v_1,\,\cdots,\, v_k)=1$ if
$[v_1,\,\cdots,\, v_k]$ is a $(k-1)$-simplex and $0$ otherwise.
\begin{theorem}
  \label{th: simp_mean}
  The mean number of $(k-1)$-simplices $N_k$ is given by
  \begin{eqnarray*}
    \E{N_{k}}=\frac{\lambda a^d(\lambda (2 \epsilon)^d)^{k-1} k^d}{k!}\cdotp
    \label{mean_final_resul}
  \end{eqnarray*}
\end{theorem}
\begin{proof}
  The number of $(k-1)$-simplices can be counted by the expression:
  \begin{equation*}
    N_k(\omega)=\frac{1}{k!}\int \varphi_k^{(d)}(v_1,\,\cdots,\, v_k)\d \omega^{(k)}(v_1,\, \cdots,\, v_k).
  \end{equation*}
  According to the Campbell-Mecke formula and since the max-distance can
  be tensorized, we have:
  \begin{align*}
    \E{N_k}&=\frac{\lambda^{k}}{k!}\int_{X^k}
    \varphi_k^{(d)}(v_1,\,\cdots,\, v_k)\d v_1\dots \d v_k\\
    &=\frac{\lambda^{k}}{k!} \left( \int_{[0,a)^k}
      \varphi_k^{(1)}(x_i,\, x_j) \d x_1 \dots \d x_k \right) ^d.
  \end{align*}
  A moment of thought reveals that for any $(x_1,\, \cdots, \, x_k)\in
  \Delta_k^{(1)}$, since $\epsilon<a/4<a/2$, there exists a unique index
  $i$ such that for all $j\in \{1,\, \cdots,\, k\}\backslash \{i\}$,
  one and only one of the two following conditions holds:
  \begin{equation*}
    x_{i}<x_j<x_{i}+2\epsilon \text{ or }  x_{i}<x_j+a<x_{i}+2\epsilon. 
  \end{equation*}
  Let $\zeta(x_1,\,\cdots,\, x_k)$ denote this index $i$. Hence, by
  invariance by translation of the Lebesgue measure,
  \begin{multline*}
    \int_{\zeta^{-1}(1)}
    \varphi_k^{(1)}(x_i,\, x_j) \d x_1 \dots \d x_k\\
    =(k-1)! \int_0^a \d x_1\int_{[x_1,\,
      x_1+2\epsilon)^{k-1}}\prod_{j=2}^{k-1} \indicator{x_j<x_{j+1}}\d
    x_2\ldots \d x_k =a(2\epsilon)^{k-1}.
  \end{multline*}
  The very same identity holds for any integral on the set
  $\zeta^{-1}(i)$ for any $i\in \{1,\, \cdots,\, k\}$ hence
  \begin{equation*}
    \int_{[0,a)^k} 
    \varphi_k^{(1)}(x_i,\, x_j) \d x_1 \dots \d x_k=ka(2\epsilon)^{k-1}.
  \end{equation*}
  The proof is thus complete.
\end{proof}
By depoissonization, we can estimate the mean number of $k$-simplices
for a Binomial process: a process with $n$ points uniformly
distributed over the torus.
\begin{corollary}
  \label{cor:depoiss1}
  The mean number $(k-1)$-simplices $N_k$ given $N_1=n$ is
  \begin{equation*}
    \E{N_k\,|\, N_1=n}=\binom{n}{k} k^d \left(\frac{2\epsilon}{a}\right)^{d(k-1)}.
  \end{equation*}
\end{corollary}
\begin{proof}
  According to Theorem~\ref{th: simp_mean}, we have:
  \begin{equation*}
    \frac{\lambda a^d(\lambda (2 \epsilon)^d)^{k-1}
      k^d}{k!}=\sum_{n=0}^\infty  \E{N_k\,|\, N_1=n} e^{-\lambda
      a^d}\frac{(\lambda a^d)^n}{n!}\cdotp
  \end{equation*}
  The principle of depoissonization is then to invert the transform
  $\Theta$ defined by:
  \begin{align*}
    \Theta \,:\, \R^\N & \longrightarrow \R[\lambda]\\
    (\alpha_n, \,n\ge 0) &\longmapsto \sum_{n\ge 0} \alpha_n
    e^{-\lambda} \frac{\lambda^n}{n!}\cdotp
  \end{align*}
  We have that $(\lambda a^d)^k=\sum_{n\geq
    k}\frac{n!}{(n-k)!}\frac{(\lambda a^d)^n}{n!}e^{-\lambda
    a^d}$. The result follows.
\end{proof}
\begin{remark}
  Considering the maximum norm simplifies the calculations.  However,
  even for the Euclidean norm, it is still possible to find a
  closed-form expression for $\E{N_2}$ and $\E{N_3}$ when we consider
  the Rips-Vietoris complex in $\mathbb{T}_a^2$. We are limited to
  small orders because no formula seems to be known for the area of
  the intersection of $k$ balls in general position. For $k=2$ and
  $3,$ the expectations are given by the following formulas:
  \begin{eqnarray*}
    \E{N_2}&=&\frac{\pi(a\lambda\epsilon)^2}{2},\\
    \E{N_3}&=&\pi\left( \pi-\frac{3\sqrt{3}}{4}\right)\frac{\lambda^3 a^2\epsilon^4}{6}\cdot
  \end{eqnarray*}
\end{remark}
Consider now the Bell's polynomial $B_d(x)$, defined as
(see~\cite{bell}):
\begin{eqnarray*}
  B_n(x)=\sum_{k=0}^{n}\stirling{n}{k}x^k,
\end{eqnarray*}
where $n$ is a positive integer and ${\stirling{n}{k}}$ is the
Stirling number of the second kind. An equivalent definition of $B_n$
can be:
\begin{eqnarray*}
  \label{eq: bell}
  B_n(x)=e^{-x}\sum_{k=0}^{\infty}\frac{x^kk^d}{k!}\cdotp
\end{eqnarray*}
These polynomials appear rather surprisingly in the computations of
the mean value of the Euler characteristic.
\begin{theorem}
  \label{th: chi_mean}
  The mean of the Euler characteristic of the simplicial complex
  $\mathcal{C}_{\epsilon}(\omega)$ is given by
  \begin{eqnarray*}
    \label{main_result_chi}
    \E{\chi}=-\left(\frac{a }{2\epsilon}\right)^de^{-\lambda(2\epsilon)^d} B_d(-\lambda(2\epsilon)^d).
  \end{eqnarray*}
\end{theorem}
\begin{proof}
  Since
  \begin{equation*}
    N_k\leq \frac{1}{k!} \prod_{j=0}^{k-1} (N_1-j)  \leq \frac{N_1^k}{k!}, \text{ then }
    \sum_{k=1}^{\infty} N_k\leq \sum_{k=1}^{\infty}\frac{N_1^k}{k!}
    \leq e^{N_1}.
  \end{equation*}
  As $\E{e^{N_1}}<\infty$, we have
  $\E{-\sum_{k=1}^{\infty}(-1)^kN_k}=-\sum_{k=1}^{\infty}(-1)^k\E{N_k}$
  and
  \begin{align*}
    \E{\chi}
    &=-\sum_{k=1}^{\infty}(-1)^k \frac{\lambda^{k}(ak(2\epsilon)^{k-1})^d}{k!}\nonumber\\
    &=\frac{a^d
      e^{-\lambda(2\epsilon)^d}}{-(2\epsilon)^d}{e^{\lambda(2\epsilon)^d}\sum_{k=0}^{\infty}\frac{(-\lambda(2\epsilon)^d)^kk^d}{k!}}\\
    &=-\left(\frac{a }{2\epsilon}\right)^de^{-\lambda(2\epsilon)^d}
    B_d(-\lambda(2\epsilon)^d).
  \end{align*}
  The proof is thus complete.
\end{proof}
If we take $d=1$, $2$ and $3$, we obtain:
\begin{align*}
  \E{\chi}&=a\lambda e^{-\lambda 2\epsilon}, \text{ for } d=1;\\
  \E{\chi}&=a^2 \lambda e^{-\lambda(2\epsilon)^2}\left(1-\lambda(2\epsilon)^2 \right), \text{ for } d=2;\\
  \E{\chi}&=a^3\lambda
  e^{-\lambda(2\epsilon)^3}\left(1-3\lambda(2\epsilon)^3+(\lambda(2\epsilon)^3)^2
  \right), \text{ for } d=3.
\end{align*}
The next corollary is an immediate consequence of
Corollary~\ref{cor:depoiss1}, obtained again by depoissonization.
\begin{corollary}
  The expectation of $\chi$ for a binomial point process with $n$
  points is given by:
  \begin{equation*}
    \E{\chi\,|\, N_1=n}=\sum_{k=0}^n (-1)^k \binom{n}{k}k^d \left(\frac{2\epsilon}{a}\right)^{d(k-1)}.
  \end{equation*}
\end{corollary}
So far, we have not say a word about Betti numbers. It turns out that
the preceding computations lead to a bound of the tail of $\beta_0$,
the number of connected components.
\begin{theorem}\label{chi_mean_thm:3}
  For $y>\lambda a^d$, we have
  \begin{eqnarray*}
    \Pr_\Lambda(\beta_0\geq y)\leq \exp\left(-\frac{y-\lambda a^d}{2}\log\left(1+ \frac{y-\lambda a^d}{(2^d-1)^2\lambda}\right) \right)\cdotp
  \end{eqnarray*}
\end{theorem}
\begin{proof}
  $\beta_0$ is the number of connected components. Since there are
  more points than connected components,
  $\E{\beta_0}\leq\E{N_1}=\lambda a^d$. According to the definition of
  $D$, $\sup_{x\in X} D_t\beta_0$ is the maximum variation of
  $\beta_0$ induced by the addition of an arbitrary point.  If the
  point $x$ is at a distance smaller than $\epsilon$ from $\omega$,
  then $D_x\beta_0\leq 0$, otherwise, $D_x\beta_0=1$, so
  $D_x\beta_0\le 1$ for any $x\in X$. Besides, this added point can
  join at most two connected components in each dimension, so in $d$
  dimensions it can join at most $2^d$ connected component, which
  means that $D\beta_0$ ranges from $-(2^d-1)$ to $1$, and then
  \begin{eqnarray*}
    \Vert D\beta_0\Vert_{L^{\infty}(L^{2}(X,\Lambda),\P)}\le \sup_\omega
     \int_{X} \vert D_x\beta_0\vert^2 \d \Lambda(x) \leq  \lambda a^d (2^d-1)^2.
  \end{eqnarray*}
  Since the function $f$ defined by \begin{equation*} f(x,y)=
    \exp\left(-\frac{k_1-x}{2k_2}\log\left(1+
        \frac{(k_1-x)k_2}{k_3y}\right) \right)\cdotp
  \end{equation*}
  is strictly increasing with respect to $x$ and $y$ for $k_1>x$, it
  follows from Theorem~\ref{th:conc_inequal} that:
  \begin{equation*}
    \Pr_\Lambda(\beta_0\geq y)\le \exp\left(-\frac{y-\lambda a^d}{2}\log\left(1+
        \frac{y-\lambda a^d}{(2^d-1)^2\lambda a^d}\right) \right),
  \end{equation*}
  for $y>\lambda a^d\ge \E{\beta_0}$.
\end{proof}

\section{Second order moments}
We now deal with the computations of the second order moments. The
proofs rely on the chaos decomposition of the number of simplices (see
Lemma \ref{lemma:f_chaos}) and the multiplication formula for iterated
integrals (see Proposition \ref{prop:chaos_prod}).  The computations
are rather technical and postponed to Appendix A. We make the
following convention: For any integer $k$,
\begin{equation*}
  \int_{X^0}\varphi_{k}^{(d)} (v_1,\,\cdots,\, v_{k}) \d v_1\dots \d v_0=\varphi_{k}^{(d)} (v_1,\,\cdots,\, v_{k}) .
\end{equation*}
\begin{lemma}
  \label{lemma:f_chaos}
  We can rewrite $N_k$ as
  \begin{eqnarray*}
    N_k=\frac{1}{k!}\sum_{i=0}^{k}{k\choose i}\lambda^{k-i}
    \ I_i\left(\int_{X^{k-i}} \varphi_{k}^{(d)} (v_1,\,\cdots,\, v_{k}) \d v_1\dots \d v_{k-i} \right).
  \end{eqnarray*}
\end{lemma}
\begin{proof}
  For $k=1$, the result is immediate with the convention made above.
  Once we have seen that
  \begin{multline*}
    \int\varphi_k^{(d)}(v_1,\dots,v_k)\d\omega^{(k)}(v_1,\, \cdots,\, v_k)\\
    =\int \left(\int \varphi_k^{(d)}(v_1,\dots,v_k)
      (d(\omega-\sum_{j=1}^{k-1}\delta(v_j))(v_k)-\d\Lambda(v_k))\right)\d\omega^{(k-1)}(v_1,\, \cdots,\, v_{k-1}) \\
    +  \int \left(\int_X \varphi_k^{(d)}(v_1,\dots,v_k) \d \Lambda
      (v_k)\right)\d\omega^{(k-1)}(v_1,\, \cdots,\, v_{k-1}),
  \end{multline*}
  the result follows by induction.
\end{proof}

\begin{theorem}
  \label{thm_cov_ksimp}
  The covariance between the number of $(k-1)$-simplices $N_k$, and
  the number of $(l-1)$-simplices, $N_l$ for $l\le k$ is given by
  \begin{multline*}
    \cov{N_k,N_l}=\sum_{i=1}^l \frac{\lambda a^d(\lambda
      (2\epsilon)^d)^{k+l-i-1}}{i!(k-i)!(l-i)!}\left(
      k+l-i+2\frac{(k-i)(l-i)}{i+1} \right)^d.
  \end{multline*}
\end{theorem}

\begin{remark}
  As for the first moment it is still possible to find, considering
  the Euclidean norm, a closed-form expression for $\var{N_k}$. We
  did not find a general expression for any dimension.
 However, when we consider the
  Rips-Vietoris complex in $\mathbb{T}_a^2$, the variance of the
  number of 1-simplices and 2-simplices are given by:
  \begin{eqnarray*}
    \var{N_2}=\left(\frac{a}{2\epsilon}\right)^2\left( \frac{\pi}{2}(4\lambda\epsilon^2)^2+\pi^2(4\lambda\epsilon^2)^3\right),
  \end{eqnarray*}
  and
  \begin{multline*}
    \var{N_3}=\left(\frac{a}{2\epsilon}\right)^2\left(
      (4\lambda\epsilon)^3\frac{\pi}{6}\left(
        \pi-\frac{3\sqrt{3}}{4}\right)+(4\lambda\epsilon^2)^4\pi\left(\frac{\pi^2}{2}-\frac{5}{12}-\frac{\pi\sqrt{3}}{2}\right)
    \right.\\ \left.  +(4\lambda \epsilon^2)^5\frac{\pi^2}{4}\left(
        \pi-\frac{3\sqrt{3}}{4}\right)^2\right)\cdotp
  \end{multline*}
\end{remark}

Since we have an expression for the variance of the number of
$k$-simplices, it is possible to calculate the variance of the Euler
characteristic.

\begin{theorem}
  \label{cor:var_chi}
  The variance of the Euler characteristic is:
  \begin{eqnarray*}
    \var{\chi}=\lambda a^d \sum_{n=1}^{\infty}c_n^d(\lambda(2\epsilon)^d)^{n-1},
  \end{eqnarray*}
  where
  \begin{multline*}
    c_n^d=\sum_{j=\lceil (n+1)/2\rceil}^n\left[2\sum_{i=n-j+1}^{j}
      \frac{(-1)^{i+j}}{(n-j)!(n-i)!(i+j-n)!}{\left(n+\frac{2(n-i)(n-j)}{1+i+j-n}
        \right)^d}\right.\\
    \left.-\frac{1}{(n-j)!^2(2j-n)!}{\left(n+\frac{2(n-j)^2}{1+2j-n}
        \right)^d}\right].
  \end{multline*}
\end{theorem}

\begin{theorem}
  \label{cor:var_chi_1d}
  In one dimension, the expression of the variance of the Euler
  characteristic is:
  \begin{eqnarray*}
    \var{\chi}={a}\left(\lambda e^{-2\lambda\epsilon}-4\lambda^2\epsilon
      e^{-4\lambda\epsilon}\right).
  \end{eqnarray*}
\end{theorem}

\begin{theorem}
  If $d=2$, we have $D\chi \le 2$ and thus
  \begin{equation*}
    \Pr(\chi-\esp{\chi}\ge x)\le
    \exp\left(-\frac{x}{4}\log\left(1+\frac{x}{2 \lambda a^2} \right) \right)\cdotp
  \end{equation*}  
\end{theorem}
\begin{proof}
  In two dimensions, according to Proposition~\ref{prop: beta_d}, the
  Euler characteristic is given by:
  \begin{equation*}
    \chi = \beta_0 - \beta_1 + \beta_2.
  \end{equation*}
  If we add a vertex on the torus, either the vertex is isolated or
  not. In the first case, it forms a new connected component
  increasing $\beta_0$ by $1$, and the number of holes,
  i.e. $\beta_1$, remains the same. Otherwise, as there is no new
  connected component, $\beta_0$ is the same, but the new vertex can
  at most fill one hole, increasing $\beta_1$ by $1$. Therefore, the
  variation of $\beta_0 - \beta_1$ is at most $1$.

  Furthermore, when we add a vertex to a simplicial complex, we know
  from Proposition~\ref{prop:chi_interpretation} that $D\beta_2\le 1$
  hence $D\chi \le 2$. Then, we use Eq.~\eqref{eq: conc_inequal} to
  complete the proof.
\end{proof}

\section{Third  order moments}
Higher order moments can be computed in a similar way but the
computations become trickier as the order increases. We here restrict
our computations to the third order moments to illustrate the general
procedure. The proof is given in Appendix B.

We are interested in the central moment, so we introduce the following
notation for the centralized number of $(k-1)$-simplices:
$\widetilde{N_k}=N_k-\E{N_k}$.

\begin{theorem}
  \label{thm_3mom}
  The third central moment of the number of $(k-1)$-simplices is given by:
  \begin{eqnarray*}
    \E{\widetilde{N_k}^3}=\sum_{i,j,s,t} 
    \frac{\lambda^{3k-i-j} t! }{(k!)^3}\binom{k}{i} \binom{k}{j} \binom{k}{s} \binom{i}{t} \binom{j}{t}
    \binom{t}{i+j-s-t} \mathcal{J}_3 (k,i,j,s,t),
  \end{eqnarray*}
  with $s \ge |i-j|$, and $\mathcal{J}_3 (k,i,j,s,t)$ is an integral
  depending on $k,i,j,s$ and $t$, defined below in \eqref{eq_chi_mean_v3:1}.
\end{theorem}

\section{Convergence}
Before going further, we must answer a natural question: Do we
retrieve the torus homology when the intensity of the Poisson process
goes to infinity, so that the number of points becomes arbitrary
large~? The answer is positive as shows the next theorem.
\begin{theorem}
  \label{chi_mean_thm:2}
  The Betti numbers of ${\mathcal C_{\epsilon}}(\omega)$ converge in
  probability to the Betti numbers of the torus as $\lambda$ goes to
  infinity:
  \begin{equation*}
    \Pr_\Lambda\left(\bigcap_{i=0}^d
      \left(\beta_i(\omega)=\binom{d}{i}
        \right)\right)\xrightarrow{\lambda\to \infty} 1,
  \end{equation*}
where $\binom{d}{i}$ is the $i$-th Betti number of the $d$-dimensional
torus, see \cite{hatcher}.
\end{theorem}
\begin{proof}
  Let $\eta<\epsilon/2$, by compactness of the torus, there exists
  ${\mathfrak B}$ a finite collection of balls of radius $\eta$
  covering $\mathbb{T}^d_{a}$. Since $\eta<\epsilon/2$, if $x$ belongs
  to some ball $B\in {\mathfrak B}$ then $B\subset B(x,\, \epsilon)$,
  hence
  \begin{equation*}
    \bigcap_{B\in {\mathfrak B} } (\omega(B)\neq 0)\subset
    \left(\U_\epsilon(\omega)=\mathbb{T}^d_{a}\right). 
  \end{equation*}
  Thus,
  \begin{align*}
    \P_\Lambda\left(\U_\epsilon(\omega)\neq
      \mathbb{T}^d_{a}\right)&\le \P_\Lambda\left(\bigcup_{B\in {\mathfrak B} }    (\omega(B)= 0)\right)\\
    &\le K \exp(-\lambda (2\eta)^d)\xrightarrow{\lambda\to \infty} 0.
  \end{align*}
  Moreover, by the nerve lemma
  \begin{equation*}
    \left(\U_\epsilon(\omega)=\mathbb{T}^d_{a}\right)\subset \bigcap_{i=0}^d\left(\beta_i(\omega)=\binom{d}{i}\right),
  \end{equation*}
  and the result follows.
\end{proof}

Let $\Gamma$ be an arbitrary connected simplicial complex of $n$
vertices.
The number of occurrences of $\Gamma$ in
$C_\epsilon(\omega)$ is denoted as $G_\Gamma(\omega)$. It must be
noted that with our construction of the simplicial complex, a complex
$\Gamma$ appears in $C_\epsilon(\omega)$ as soon as its edges are in
$C_\epsilon(\omega)$.  The set of edges of $\Gamma$, denoted by
$J_\Gamma$ is a subset of $\{1,\, \dots,\, n\}\times \{1,\, \dots,\,
n\}$.  Let us define the following function on the vertices
of~$\Gamma$:
\begin{equation*}
  \widetilde{h^\Gamma}(v_1,\dots,\, v_n)=\frac{1}{c_{\Gamma}}\prod_{(i,j)\in J_\Gamma}\indicator{\rho_d(v_i,v_j)<\epsilon},  
\end{equation*}
where $c_{\Gamma}$ is the number of permutations $\sigma$ of $\{v_1,\,
\dots,\, v_n\}$ such that
\begin{equation*}
  \widetilde{h^\Gamma}(v_1,\, \dots,\, v_n)=\widetilde{h^\Gamma}(v_{\sigma(1)},\, \dots,\, v_{\sigma(n)}),
\end{equation*}
and let $f^\Gamma$ be the symmetrization of
$\widetilde{h^\Gamma}$. Then, we have:
\begin{equation}\label{chi_mean_eq:2}
  G_\Gamma(\omega)=\int f^\Gamma(v_1,\dots,v_n)
  \d\omega^{(n)}(v_1,\, \cdots,\, v_n). 
\end{equation}


\begin{lemma}
  \label{lemma: chaos}
  The random variable $G_\Gamma$ has a chaos representation given by:
  \begin{eqnarray*}
    G_\Gamma=\sum_{i=0}^n I_i(f^\Gamma_i),
  \end{eqnarray*}
  where $f^\Gamma_i$ is the bounded symmetric function defined as
  \begin{equation}
    \label{eq: f_i}
    f^\Gamma_i(v_{i+1},\,\dots,\, v_n)={n\choose i}\lambda^{n-i}\int\limits_{X^{n-i}}
    f^\Gamma(v_1,\dots,v_n) \d v_1\dots  \d v_{n-i}.
  \end{equation}
\end{lemma}
\begin{proof}
  From \eqref{chi_mean_eq:2}, using the binomial expansion and some
  algebra, we obtain
  \begin{multline*}
    G_\Gamma(\omega)\\=\sum_{i=0}^{n}\int\left({n\choose
        i}\int\limits_{X^{n-i}} f^\Gamma(v_1,\dots,v_n)\lambda \d
      v_1 \dots \lambda \d
      v_{n-i}\right)\d\omega_\Lambda^{(i)}(v_{n-i+1},\, \cdots,\, v_n).
  \end{multline*}
  We define, for any $i\in \{1,\,\dots,\, n\}$,
  \begin{equation*}
    f^\Gamma_i(v_{i+1},\,\dots,\, v_n)={n\choose i}\lambda^{n-i}\int_{X^{n-i}}
    f^\Gamma(v_1,\dots,v_n) \d v_1\dots  \d v_{n-i}.
  \end{equation*}
  To conclude the proof, we note that, since $X$ is a bounded set and
  $h^\Gamma$ is bounded, $f^\Gamma_i$ is bounded.
\end{proof}

\begin{theorem}
  \label{prop: reminder}
  There exists $c>0$ such that for any~$\lambda \ge 1$,
  \begin{eqnarray*}
    d_{W}\left(\dfrac{G_\Gamma-\E{G_\Gamma}}{\sqrt{\var{G_\Gamma}}},\ \mathcal{N}(0,1)\right)\leq \frac{c}{\lambda^{1/2}}\cdotp
  \end{eqnarray*}
\end{theorem}
\begin{proof}
  Let $F=\dfrac{G_\Gamma-\E{G_\Gamma}}{\sqrt{\var{G_\Gamma}}}$.
  Provided that $\Gamma$ has $n$ vertices, according to
  Lemma~\ref{lemma: chaos}, we have the following identities:
  \begin{align*}
    D_tF&=\frac{1}{\sqrt{{\var{G_\Gamma}}}}\sum_{i=1}^{n}iI_{i-1}(f_i^\Gamma(*,\,
    t)),\\
    -D_tL^{-1}F&=\frac{1}{\sqrt{\var{G_\Gamma}}}\sum_{i=1}^{n}
    I_{i-1}(f_i^\Gamma(*,\, t)),\\
    \var{G_\Gamma}&=\sum_{i=1}^{n}i!\ \left\Vert
      f_{i}^\Gamma\right\Vert_{L^2(X,\Lambda)^{\otimes i}}^2 .
  \end{align*}
  Hence, $\var{G_\Gamma}$ is a polynomial of degree $2n-1$ with
  respect to $\lambda$. From Proposition~\ref{prop:chaos_prod}, it is
  tedious but straightforward to see that $\langle DL^{-1}F,\, DF
  \rangle_{L^2(X,\Lambda)}$ is a polynomial of degree $2n-2$ with
  random coefficients depending on the integrals over $X^i$ of the
  $f_i^\Gamma$. According to Lemma \ref{lemma: chaos}, these
  coefficients are bounded almost-surely.
  Hence there exists a constant $c>0$ such that
  \begin{equation*}
    \E{\left| 1+ \langle DL^{-1}F,\, DF \rangle_{L^2(X,\Lambda)}\right| }\le c\, \lambda^{-1/2}.
  \end{equation*}
  The same kind of computations shows that
  \begin{equation*}
    \int_X\E{\vert D_xF\vert^2\vert D_xL^{-1}F\vert}\lambda \d x \le c\, \lambda^{-1/2}.
  \end{equation*}
  Then, the result follows from Theorem~\ref{thm_chi_mean_v2:1}.
\end{proof}

\providecommand{\bysame}{\leavevmode\hbox to3em{\hrulefill}\thinspace}
\providecommand{\MR}{\relax\ifhmode\unskip\space\fi MR }
\providecommand{\MRhref}[2]{%
  \href{http://www.ams.org/mathscinet-getitem?mr=#1}{#2} }
\providecommand{\href}[2]{#2}

\appendix
\section{Proofs of the second order moments}
\subsection{Proof of Theorem \ref{thm_cov_ksimp}}
By Lemma~\ref{lemma:f_chaos}, we can rewrite the covariance between
$N_k$ and $N_l$ with $l \leq k$:
\begin{eqnarray*}
  \cov{N_k,N_l}&=&\E{(N_k-\E{N_k})(N_l-\E{N_l})}\\
  &=&    \E{\frac{1}{k!}\sum_{i=1}^{k}{k\choose i}\lambda^{k-i}  I_i\left(f_i^k\right) 
    \frac{1}{l!}\sum_{j=1}^{l}{l\choose j}\lambda^{l-j} I_j\left(f_j^l\right)},
\end{eqnarray*}
where
\begin{eqnarray*}
  f_i^k (v_{k-i+1},\dots,v_{k})= \int_{X^{k-i}} \varphi_k^{(d)}(v_1,\dots,v_k) \d v_1\dots \d v_{k-i}.
\end{eqnarray*}
Using the isometry formula, given by Eq.~\eqref{eq: isometry}, we
have:
\begin{eqnarray*}
  \cov{N_k,N_l}&=& \frac{1}{k!l!}\sum_{i=1}^{l}{k\choose i}{l\choose i}\lambda^{k+l-2i} \E{I_i\left(f_i^k\right) I_i\left(f_i^l\right)} \\
  &=&\frac{1}{k!l!}\sum_{i=1}^{l}{k\choose i}{l\choose
    i}\lambda^{k+l-2i}i!\langle f_i^k, f_i^l\rangle_{L^2(X,\Lambda)^{\circ i}}\\
  &=&\sum_{i=1}^l \frac{1}{i!(k-i)!(l-i)!} \lambda^{k+l-2i} \langle
  f_i^k, f_i^l\rangle_{L^2(X,\Lambda)^{\circ i}}
  \label{eq:cov} 
\end{eqnarray*}
Hence, we are reduced to compute
\begin{multline*}
  \langle f_i^k , f_i^l \rangle_{L^2(X,\Lambda)^{\circ i}}=
  \int_{X^i} \left(\int_{X^{l-i}} \!\!
  \varphi_{l}^{(d)}(v_1,\dots,v_l)\d v_{i+1}\! \dots \! \d v_l\right)
  \\ \times \left(\int_{X^{k-i}} \!\! \varphi_k^{(d)}(v_1,\dots,v_k)
  \d v_{i+1}\!\dots\! \d v_{k}\right)\  \lambda \!\! \d v_{1}\dots \lambda \!\!
  \d v_{i}.
\end{multline*}
Let us denote $\mathcal{J}_2(m_1,m_2,m_{12})$ the integral on two
simplices of respectively $m_1+m_{12}$ and $m_2+m_{12}$ vertices with
$m_{12}>0$ common vertices:
\begin{multline*}
  \mathcal{J}_2(m_1,m_2,m_{12})=\\\int_{X^M}
  \varphi_{m_1+m_{12}}^{(d)} (v_1,\dots,
  v_{m_1+m_{12}})\varphi_{m_2+m_{12}}^{(d)} (v_{m_1+1},\dots,v_M)\d
  v_1 \dots \d v_M,
\end{multline*}
with $M=m_1+m_2+m_{12}$. Then we can rewrite:
\begin{eqnarray*}
  \langle f_i^k, f_i^l \rangle_{L^2(X,\Lambda)^{\circ i}}=\lambda^i \mathcal{J}_2(l-i,k-i,i),
\end{eqnarray*}
and it then remains to compute $\mathcal{J}_2(m_1,m_2,m_{12})$.

First, thanks to the tensorization property of the max-distance, we can write:
\begin{multline*}
  \mathcal{J}_2(m_1,m_2,m_{12})=\\\left( \int_{[0,\, a)^M}
    \varphi_{m_1+m_{12}}^{(1)} (x_1,\dots, x_{m_1+m_{12}})
    \varphi_{m_2+m_{12}}^{(1)} (x_{m_1+1},\dots,x_M)\d x_1 \dots \d
    x_M \right)^d
\end{multline*}
Let us split the integration domain of $\mathcal{J}_2$ into two parts:
\begin{itemize}
\item $A_1=\{ (x_1,\dots,x_M) \in \Delta_M^{(1)},\
  \varphi_M^{(1)}(x_1,\dots,x_M)=1 \}$, we recognize the integral
  calculated in the proof of Theorem \ref{th: simp_mean}:
  \begin{multline*}
    \int_{A_1} \varphi_{m_1+m_{12}}^{(1)} (x_1,\dots, x_{m_1+m_{12}})\varphi_{m_2+m_{12}}^{(1)} (x_{m_1+1},\dots,x_M)\d x_1 \dots \d x_M \\
    =M (2 \epsilon)^{M-1} a.
  \end{multline*}
\item $A_2=\{ (x_1,\dots,x_M) \in \Delta_M^{(1)},\ 
  \varphi_M^{(1)}(x_1,\dots,x_M)\neq1 \}$.
\end{itemize}

As in the proof of Theorem \ref{th: simp_mean}, we denote
$\zeta(x_1,\,\cdots,\, x_M)$ the index $i$ such that $
x_{i}<x_j<x_{i}+2\epsilon \text{ or }
x_{i}<x_j+a<x_{i}+2\epsilon$, which exists since $\epsilon < a/4$ and $m_{12}>0$. By
symmetry, we can reduce the analysis to the situation where  $\zeta(x_1,\,\cdots,\, x_M)=1$ and $x_1$ pertains to the first simplex of $m_1+m_{12}$.  We then order the
three sets of vertices such that:
\begin{eqnarray*}
  x_1< \dots < x_{m_1}, \quad
  x_{m_1+1}<\dots <x_{m_1+m_{12}}, \text{ and }
  x_{m_1+m_{12}+1}<\dots <x_{M}.
\end{eqnarray*}
Since $(x_1,\dots,x_M)$ belongs to $A_2$, we have $x_M-x_1>2\epsilon$.

Let us denote 
$J_{a}(f)(x)=\int_x^a f(u) \d u$ and by induction  $$J_{a}^{(m)}(f)(x) = \int_x^a J_a^{(m-1)}(f)(u)\d u.$$
Then we have by invariance by translation of the Lebesgue measure,
\begin{multline*}
  \int_{A_2} \varphi_{m_1+m_{12}}^{(1)} (x_1,\dots,
  x_{m_1+m_{12}})\varphi_{m_2+m_{12}}^{(1)} (x_{m_1+1},\dots,x_M)\d
  x_1 \dots \d x_M\\ = 2 m_1! m_2! m_{12}!\!\!\!  \int_0^a \!\!\!\!\!
  J_{x_1+2\epsilon}^{(m_1-1)} (\mathbf{1}) (x_1)\!\!\!
  \int_{x_1+2\epsilon}^{x_1+4\epsilon} \!\!\!\!\!\!\!\!\!\!\!\!\!\!
  (-J_{x_M-2\epsilon}^{(m_2-1)} (\mathbf{1}) (x_M))
  ~J_{x_1+2\epsilon}^{(m_{12})} (\mathbf{1}) (x_M-2\epsilon) \d x_M\!
  \d x_1.
\end{multline*}
We easily find that:
\begin{eqnarray*}
  J_{x_1+2\epsilon}^{(m_1-1)} (\mathbf{1}) (x_1) &=& \frac{(2\epsilon)^{m_1-1}}{(m_1-1)!}\\
  -J_{x_M-2\epsilon}^{(m_2-1)} (\mathbf{1}) (x_M) &=& \frac{(2\epsilon)^{m_2-1}}{(m_2-1)!},\\
  J_{x_1+2\epsilon}^{(m_{12})} (\mathbf{1}) (x_M-2\epsilon) &=& \frac{(x_1-x_M+4\epsilon)^{m_{12}}}{m_{12}!}\cdotp
\end{eqnarray*}
Thus we have:
\begin{multline*}
  \int_{A_2} \varphi_{m_1+m_{12}}^{(1)} (x_1,\dots,
  x_{m_1+m_{12}})\varphi_{m_2+m_{12}}^{(1)} (x_{m_1+1},\dots,x_M)\d
  x_1 \dots \d x_M \\
  =\frac{2m_1m_2}{m_{12}+1}(2\epsilon)^{M-1}a.
\end{multline*}
Then,
\begin{eqnarray*}
  \mathcal{J}_2(m_1,m_2,m_{12})=(m_1+m_2+m_{12} +\frac{2m_1m_2}{m_{12}+1})^da^d(2\epsilon)^{(m_1+m_2+m_{12}-1)d}
\end{eqnarray*}
concluding the proof.

\subsection{Proof of Theorem \ref{cor:var_chi}}
The variance of $\chi$ is given by:
\begin{eqnarray*}
  \var{\chi}&=&\E{(\chi-\E{\chi})^2}\\
  &=&\E{\left( \sum_{k=1}^{\infty}(-1)^k(N_k-\E{N_k} \right)^2}\\
  &=&\E{\sum_{i=1}^{\infty}\sum_{j=1}^{\infty}(-1)^{i+j}(N_i-\E{N_i})(N_j-\E{N_j})}.
\end{eqnarray*}
We remark that $N_i \leq {N_1^i}/{i!}$, thus
\begin{eqnarray*}
  &&\E{\sum_{i=1}^{\infty}\sum_{j=1}^{\infty} \vert (N_i-\E{N_i})(N_j-\E{N_j}) \vert }\\
  &&\le \E{\sum_{i=1}^{\infty}\sum_{j=1}^{\infty} N_iN_j
    +\E{N_i}\E{N_j} + N_i\E{N_j}+N_j\E{N_i}}\\ 
  &&\le \E{\sum_{i=1}^{\infty}\sum_{j=1}^{\infty} 
    \frac{N_1^{i+j}+\E{N_1^i}\E{N_1^j}+N_1^i\E{N_1^j}+N_1^j\E{N_1^i}}{i!j!}}\\
  && \le \E{e^{2N_1}+e^{2\E{N_1}}+2e^{N_1+\E{N_1}}}\\
  && < \infty.
\end{eqnarray*}
Thus, we can write
\begin{eqnarray*}
  \var{\chi}=\sum_{i=1}^{\infty}(-1)^i\sum_{j=1}^{\infty}(-1)^{j}\cov{N_i,N_j}.
\end{eqnarray*}
The result follows by Theorem \ref{thm_cov_ksimp}.

\subsection{Proof of Theorem \ref{cor:var_chi_1d}}
If $d=1$, according to Theorem~\ref{cor:var_chi}:
\begin{eqnarray}
  \label{eq:var_chi}
  \var{\chi}=\frac{a}{2\epsilon}\sum_{n=1}^{\infty}c_n^1(2\lambda \epsilon)^n.
\end{eqnarray}
Moreover, we define
\begin{eqnarray*}
  \alpha_n=\sum_{j=\left\lceil \frac{n+1}{2}\right\rceil}^n\left[2\sum_{i=n-j+1}^{j}
    \frac{(-1)^{i+j}n}{(n-j)!(n-i)!(i+j-n)!}-\frac{n}{(n-j)!^2(2j-n)!}\right],
\end{eqnarray*}
and $\beta_n= c_n^1-\alpha_n$.  It is well known that
\begin{eqnarray*}
  \sum_{i=0}^{2j-n}(-1)^i{j\choose i}=(-1)^{2j-n-1}{j-1\choose 2j-n},  
\end{eqnarray*}
using Stiffel's relation, we obtain:
\begin{eqnarray}
  \label{eq:alpha}
  \alpha_n &=& (-1)^n\frac{n}{n!}\sum_{j=\left\lceil \frac{n+1}{2}\right\rceil}^n\left[{n\choose j}2\sum_{i=0}^{2j-n}
    (-1)^i{j\choose i}+2(-1)^n{n\choose j}   \right]\nonumber\\
  &=&\frac{1}{(n-1)!}\sum_{j=\left\lceil \frac{n+1}{2}\right\rceil}^n\left[2{n\choose j}{j-1\choose n-j-1}-{n\choose j}{j\choose n-j}-2(-1)^n{n\choose j} \right]\nonumber\\
  &=&\frac{1}{(n-1)!}\sum_{j=\left\lceil \frac{n+1}{2}\right\rceil}^n\left[ {n\choose j} \left({j-1\choose n-j}-{j-1\choose n-j-1}\right)-2(-1)^n{n\choose j} \right].
\end{eqnarray}
The identity ${n\choose j}={n\choose n-j}$ allows us to write that
\begin{gather*}
  \sum_{j=\lceil (n+1)/2
    \rceil}^n (-2(-1)^n){n\choose j}=\sum_{j=0}^n {n\choose j}=2^n,\ \ n\mbox{ odd,}\\
  \sum_{j=\lceil (n+1)/2 \rceil}^n (-2(-1)^n){n\choose j}={n\choose
    n/2}+\sum_{j=0}^n -{n\choose j}=-2^n+{n\choose n/2},\ \ n\mbox{
    even.}
\end{gather*}
Since ${j-1\choose n-j}=0$ for $j<\left\lceil
  \frac{n+1}{2}\right\rceil$, we have
\begin{multline*}
  \sum_{j=\left\lceil \frac{n+1}{2}\right\rceil}^n{n\choose j}
  \left({j-1\choose n-j}-{j-1\choose n-j-1}\right)\\
  =\sum_{j=1}^n{n\choose j} \left({j-1\choose n-j}-{j-1\choose
      n-j-1}\right) -\binom{n}{n/2} \frac{1+(-1)^n}{2}
\end{multline*}
By known formulas on hypergeometric functions, we have that:
\begin{multline*}
  \sum_{j=\left\lceil \frac{n+1}{2}\right\rceil}^n{n\choose j}
  \left({j-1\choose n-j}-{j-1\choose n-j-1}\right) =(-1)^{n+1}
  -\binom{n}{n/2} \frac{1+(-1)^n}{2}
\end{multline*}
Then, we substitute these last two expressions in Eq.~\eqref{eq:alpha}
to obtain
\begin{eqnarray*}
  \alpha_n=(-1)^n\frac{(1 -2^n)\indicator{n\ge 1}}{(n-1)!},
\end{eqnarray*}
and thus
\begin{eqnarray*}
  \sum_{i=0}^{\infty}\alpha_nx^n=-x e^{-x}+2x e^{-2x}.
\end{eqnarray*}
Proceeding along the same line, $\beta_n$ is given by
\begin{eqnarray*}
  \beta_n&=&\sum_{j=\left\lceil \frac{n+1}{2}\right\rceil}^n\left[2\sum_{i=n-j+1}^{j}
    \frac{(-1)^{i+j}2(n-i)(n-j)}{(n-j)!(n-i)!(i+j-n+1)!}\right.\\ 
  &&\mbox{\hspace{6cm}}\left.-\frac{2(n-j)^2}{(n-j)!^2(2j-n+1)!}\right]\\
  &=&(-1)^n\left(\frac{(-2+2^n)\indicator{n\ge1}}{(n-1)!}-\frac{2\indicator{n\ge 2}}{(n-2)!} \right),
\end{eqnarray*}
and again we can simplify the power series
$\sum_{n=0}^{\infty}\beta_nx^n$ as
\begin{eqnarray*}
  \sum_{n=0}^{\infty}\beta_nx^n=2x e^{-x}-2(x+x^2) e^{-2x}.
\end{eqnarray*}
Then, substituting $\alpha_n$ and $\beta_n$ in Eq.~\eqref{eq:var_chi}
yields the result.

\section{Proof of the third order moment}

\subsection{Proof of Theorem \ref{thm_3mom}}
From Lemma ~\ref{lemma:f_chaos} , we know that the chaos decomposition
of the number of $(k-1)$-simplices is given by
\begin{eqnarray*}
  \widetilde{N_k}=\sum_{i=1}^k I_i(h_i),
\end{eqnarray*}
with
\begin{eqnarray*}
  h_i(v_1,\dots,v_i)=\frac{1}{k!}\binom{k}{i} \lambda^{k-i} \int_{X^{k-i}} \varphi_{k}^{(d)} (v_1,\dots,v_k)\d v_{i+1}\dots \d v_{k},
\end{eqnarray*}
and
\begin{eqnarray*}
  I_i(h_i)=\int_{X^i} h_i (v_1,\dots,v_i)\d \omega_{\Lambda}^{(i)}(v_1,\dots,v_i).
\end{eqnarray*}
Then, we define denoting $u=i+j-s$,
$$g_{i,j,s,t}= t! \binom{i}{t} \binom{j}{t}
\binom{t}{u-t}h_i \circ_t^{u-t} h_j$$ and using the chaos expansion
(cf Proposition~\ref{prop:chaos_prod}), we have
\begin{eqnarray*}
  \widetilde{N_k}^3&=&(\sum_{i=1}^k I_i(h_i))^3\\
  &=&\left( \sum_{i=1}^{k} \sum_{j=1}^{k} I_i(h_i)I_j(h_j) \right) \left( \sum_{l=1}^kI_l(h_l) \right)\\
  &=&\sum_{i,j,l=1}^{k} \sum_{s=|i-j|}^{i+j} \sum_{t=\lceil \frac{u}{2}
    \rceil}^{u\land i \land j} I_s(g_{i,j,s,t}) I_l(h_l).
\end{eqnarray*}
According to~\eqref{eq: isometry}, we obtain:
\begin{eqnarray*}
  \E{\widetilde{N_k}^3} &=& \E{\sum_{i,j=1}^{k} \sum_{s=|i-j|\vee
      1}^{i+j \wedge k} \sum_{t=\lceil \frac{u}{2}
      \rceil}^{u\land i \land j} I_s(g_{i,j,s,t}) I_s(h_s)}\\
  &=&\!\!\!\!\!\!\! \sum_{i,j=1}^{k} \sum_{s=|i-j|\vee 1}^{i+j\wedge k} \sum_{t=\lceil \frac{u}{2}
    \rceil}^{u\land i \land j} \int_{X^s} g_{i,j,s,t} h_s \lambda^s \d v_1\dots \d v_s\\
  &=&\!\!\!\!\!\!\! \sum_{i,j=1}^{k} \sum_{s=|i-j|\vee 1}^{i+j\wedge k}  \sum_{t=\lceil \frac{u}{2} \rceil}^{u \land i \land j} 
  \!\! \lambda^s t!\binom{i}{t} \binom{j}{t} \binom{t}{u-t} \!\!\! \int_{X^s} \! (h_i
  \circ_t^{u-t} h_j) h_s \d v_1 \!\dots \! \d v_s.
\end{eqnarray*}

We denote $\mathcal{J}_3 (k,i,j,s,t)$ the following integral:
\begin{multline}\label{eq_chi_mean_v3:1}
  \mathcal{J}_3(k,i,j,s,t)=
  \int_{X^{3k-t-s}} 
  \varphi_{k}^{(d)} (v_1,\dots, v_{k})
  \varphi_{k}^{(d)} (v_{1},\dots, v_{t}, 
  v_{k+1},\dots,  v_{2k-t})\\
  \varphi_{k}^{(d)} (v_1,\dots, v_{2t-i-j+s},
  v_{t+1},\dots,v_{i},v_{2k-j+1},\dots,v_{3k-t-s})
  \d v_1 \dots \d v_{3k-t-s},
\end{multline}
for $i,j,s$ and $t$ bounded as in the previous sums. We recognize the
integral on three $(k-1)$-simplices with $u-t$, $i-t$, and $j-t$
common vertices to only two of them, and $2t-u$ common vertices to the
three of them. Then, we can write:
\begin{multline*}
  \E{\widetilde{N_k}^3} =\sum_{i,j,=1}^{k} \sum_{s=|i-j|\vee
    1}^{i+j\wedge k} \sum_{t=\lceil \frac{u}{2} \rceil}^{u \land i
    \land j} \frac{\lambda^{3k-i-j} t!}{(k!)^3} \binom{k}{i}
  \binom{k}{j} \binom{k}{s} \binom{i}{t} \binom{j}{t}
  \binom{t}{u-t}\\
  \mathcal{J}_3 (k,i,j,s,t).
\end{multline*}
Finally, relaxing the boundaries on the sums conclude the proof.

\end{document}